\documentclass[12pt]{article}
\usepackage{graphicx}
\usepackage{amsmath}
\usepackage{amsfonts}
\usepackage{amsthm}

\newtheorem{theorem}{Theorem}[section]
\newtheorem{lemma}[theorem]{Lemma}

\begin{document}

\title{Autobiographical Numbers}

\author{Tanya Khovanova}
\date{March 2, 2008}
\maketitle

\begin{abstract}
I introduce autobiographical numbers as defined in A046043 (see Online Encyclopedia of Integer Sequences \cite{OEIS}). I continue by defining and analyzing biographies, curricula vitae and complete life stories of numbers. I end with the definition of mutually-praising number pairs.
\end{abstract}


\section{Definition}

Do you know that 1210 is the smallest autobiographical number? You probably do not know what an autobiographical number is. You might well think that such a number should be a pompous self-centered figure whose only purpose in life is to describe itself. 

Here is the formal definition: an autobiographical number is a number $N$ such that the first digit of $N$ counts how many zeroes are in $N$, the second digit counts how many ones are in $N$ and so on. In our example, 1210 has 1 zero, 2 ones, 1 two and 0 threes. 

\section{All Autobiographical Numbers}

Let us find all autobiographical numbers using the ``zoom-in'' method.

\begin{enumerate}
\renewcommand{\labelenumi}{\alph{enumi}.}
\item By definition, the autobiographies can not have more than 10 digits. It is nice to know that these egotistical numbers cannot be too grand.
\item The sum of the digits in an autobiography equals the number of the digits. Consequently, the sum of the digits is not more than 10.
\item The first digit is the number of zeroes. As you know, self-respecting integers do not start with a zero. Hence, the number of zeroes is not zero.
\item Subtracting statement ``c'' from statement ``b'' above, we get a resulting statement that the sum of all the digits, except for the first one, is equal to the number of non-zero digits plus 1.
\item This means, other than the first digit, the set of all other non-zero digits consists of several ones and 1 two.
\item In particular, the number of ones is either 0, 1 or 2. 
\end{enumerate}

Now we continue zooming-in in three different directions depending on the number of ones. In this paper, I will consider only the case in which there are no ones; I leave the other two cases to the reader. 

\begin{itemize}
\item If the number of ones is zero, then the only non-zero non-first digit of such a number is 2.
\item This 2 should be included in the autobiography; since the third digit of the number is not zero, it must be 2.
\item The number has 2 twos.
\item It must be 2020.
\end{itemize}

Here is the full set of autobiographical numbers: 1210, 2020, 21200, 3211000, 42101000, 521001000, 6210001000.

This is the sequence A104786 in the Online Encyclopedia of Integer Sequences (OEIS) \cite{OEIS}, where I first encountered the autobiographical numbers.

\section{Biographies}

Consider this: if there is a notion of an autobiography of a number, then it would be logical to expect there to be a notion of a biography of a number. What would be the logical candidate for a biography of a number? Let us say that given a number $N$, its biography is another number $M$ such that the first digit of $M$ is the number of zeroes in $N$, the second digit of $M$ is the number of ones in $N$ and so on.

Of course, for a number to have a biography, we need to assume that none of its digits are present more than nine times. There are several problems with the definition of a biography.

The first problem is that if $N$ does not have zeroes, its biography starts with a zero. As numbers do not start with 0, that biography is not a number! Alternatively, if $N$ starts with 0, it can have a biography but $N$ is not a number. Luckily for this article, a digit string starting with zeroes cannot be an autobiographical string, because the number of zeroes is not zero. It is a relief that those illegitimate strings that are trying to pretend to be numbers cannot actually be autobiographical.

The second problem with biographies is that a number can have many biographies. Indeed, if a number does not have nines, you can remove or add zeroes at the end of a biography to get another biography of the same number. Since mathematicians like to define things uniquely, we might consider it a disadvantage if a number has several biographies. In real life it is possible to have many biographies of a person. So the second problem is not a big problem. I will call the shortest possible biography of a number \textit{the curriculum vitae (CV)} and the longest possible biography \textit{the complete life story (CLS)}.

The third problem is that numbers with the same digits in different permutations have the same biographies. So, in a sense, a biography follows the life not of a number, but rather the set of its digits.

\section{Curricula Vitae}

Suppose for now we allow a biography to start with 0. Also, let us choose the curriculum vitae --- the shortest biography, if there are several. Let us build a sequence of CVs, where an element in the sequence is the CV of the previous one. As an example, we start with 0. Zero's CV is 1, one's CV is 01, continuing that we get the following sequence: 0, 1, 01, 11, 02, 101, 12, 011, 12, 011, 12, $\ldots$. You can see that the CVs' sequence fell into a cycle in this case. I tried sequences of CVs starting with numbers up to 10,000,000. I found that they fall into two cycles. One cycle is described above, and the other one is: 22, 002, 201, 111, 03, 1001, 22. Can we find another cycle or, alternatively, can we prove that all sequences of CVs converge to only these two cycles?

Before discussing the cycles let us discuss numbers that allow an infinite sequence of CVs. Given a number $N$ let us denote its sequence of CVs as $CV_n(N)$. That is, $CV_n(N)$ is the curriculum vitae of $CV_{n-1}(N)$. I will assume that the sequence of CVs starts with index 0: $CV_0(N) = N$. We know that not every number can have a CV. If a number has a digit repeated more than 9 times its CV cannot be defined. What about the sequence of CVs? Can we describe the numbers $N$ that allow us to define an infinite sequence of CVs, $CV_n(N)$?

\begin{lemma}\label{CVSeqdefined} Numbers for which an infinite sequence of CVs cannot be defined are of three categories:
\begin{itemize}
\item Numbers that have a digit repeated more than 9 times
\item Numbers that have all the digits from 0 to 9 repeated the same number of times
\item Numbers that have their digits with the set of multiplicities defined by the set  $\{0, 1, 2, 3, 4, 5, 6, 7, 8, 9\}$.
\end{itemize}
\end{lemma}

\begin{proof}
It is easy to see that for $n > 0$, $CV_n(N)$ has no more than 10 digits, and, consequently, for $n > 1$, $CV_n(N)$ has a sum of digits no bigger than 10. This means, the only CV that can not have its own CV must have 10 digits, and they must all be the same. Let us denote this CV by $xxxxxxxxxx$, where $x$ is a digit between 1 and 9 (obviously, a CV cannot consist of zeroes only). If the number $xxxxxxxxxx$ is a CV of some number $M$, then $M$ has all the digits from 0 to 9 repeated $x$ times. Can $M$ itself be a CV? Only if it has no more than 10 digits. This means, only if $x = 1$, the number $xxxxxxxxxx$ can be a CV of a CV. In this case $M$ has $\{0, 1, 2, 3, 4, 5, 6, 7, 8, 9\}$ as its set of digits. 

We see that numbers with a digit repeated more than 9 times cannot have a CV. Also, numbers with all the digits from 0 to 9 repeated the same number of times cannot have a CV of a CV. In addition, numbers that have their digits with the set of multiplicities as $\{0, 1, 2, 3, 4, 5, 6, 7, 8, 9\}$ cannot have a CV of a CV of a CV. All other numbers allow us to define an infinite sequence of CVs.
\end{proof}

From now on we will consider only numbers $N$ for which the infinite sequence of CVs exists. We can use the methods from the proof above for the following lemma:

\begin{lemma}\label{nosamedigits}
If, for $n > 2$, $CV_n(N)$ has all the same digits, its length can not be more than 5.
\end{lemma}

\begin{proof}
If a CV has all the same digits, this digit cannot be zero by definition. If a CV of a CV with all the same digits has length more than five, then this digit cannot be more than 1, because the sum of the digits of a CV of a CV cannot be more than 10. Suppose a CV of a CV of some number $M$ has at least 6 ones. What is the smallest sum of digits that a CV of $M$ has? Obviously, the smallest sum is when the CV of $M$ consists of the digits 0, 1, 2, 3, 4 and 5. In this case the sum of digits of the CV of $M$ is more than 10. This means that the CV of $M$ cannot be a CV of a CV. Therefore, $M$ can not be a CV itself.
\end{proof}

Let us define a \textit{height} of a CV as a maximum of two numbers: the largest digit plus 1 or the number of digits.

\begin{lemma}
Starting with $n > 2$, if the height of $CV_n(N)$ is more than 5, it cannot go up.
\end{lemma}

\begin{proof}
The largest digit of the CV of a number $N$ cannot be more than the number of digits of $N$. The number of digits of the CV of $N$ is the largest digit of $N$ plus 1. Hence, the only case when the height of the CV of $N$ can be bigger than the height of $N$ is when $N$ has all the same digits, say $a$, and the number of digits is more than $a$. By lemma \ref{nosamedigits}, starting with $n > 2$, if $CV_n(N)$ has all the same digits, the number of digits cannot be more than 5.
\end{proof}

A stronger fact is true. Namely, in the sequence of CVs, the height goes down eventually until it is less than 6. To prove this we need to do a finer analysis than in the previous proof. Suppose for some large $n$, the height of all the consecutive CVs stabilizes at some number $h$ bigger than 5. I will consider 3 cases:

\begin{enumerate}
\item The largest digit of $CV_n(N)$ is $h-1$, and $CV_n(N)$ has $h$ digits (both the number of digits and the largest digit define the height).
\item The largest digit of $CV_n(N)$ is less than $h-1$, and $CV_n(N)$ has $h$ digits (the number of digits defines the height). 
\item The largest digit of $CV_n(N)$ is $h-1$, and $CV_n(N)$ has less than $h$ digits (the largest digit defines the height). 
\end{enumerate}

If $CV_n(N)$ belongs to the second case, then $CV_{n+1}(N)$ has less than $h$ digits. For $CV_{n+1}(N)$ to have the height $h$, one of the digits of $CV_{n+1}(N)$ has to be  $h-1$. That means $CV_n(N)$ has one of its digits repeated $h-1$ times. This digit could be either zero or one. In this case, we can calculate $CV_{n+2}(N)$ explicitly, and it will have a lower height. That means, by our assumption that the height is stable, that $CV_n(N)$, as well as all the consecutive CVs, can not belong to the second case.

If $CV_n(N)$ belongs to the third case, then $CV_{n+1}(N)$ belongs to the second case. Hence, by the assumption, $CV_n(N)$, as well as all the consecutive CVs, should belong to the first case.

If $CV_n(N)$ belongs to the first case, then $CV_{n+1}(N)$ stays in the first case only if $CV_n(N)$ consists of one digit $h$, and all other digits are the same. That same digit could be either 0 or 1, and in both cases, $CV_{n+2}(N)$ is $(h-2)10000..1$ and has lower height. 

Hence, the height goes down.

Thus, we proved that starting from a big index the sequence of CVs has heights not more than 5. From here we can use a computer to check all the numbers up to 100,000, which I did. I leave it as an exercise to prove that there are two cycles without a computer. Thus, I actually proved the theorem:

\begin{theorem}
If a number allows an infinite sequence of CVs, this sequence falls eventually into one of two cycles:
\begin{itemize}
\item 12, 011
\item 22, 002, 201, 111, 03, 1001
\end{itemize}
\end{theorem}

Now let us move on to broader biographies --- complete life stories.

\section{Complete Life Story}

Let us build the sequence of complete biographies, that is, life stories, starting with 0: 0, 1000000000, 9100000000, 8100000001, 7200000010, 7110000100, 6300000100, 7101001000, 6300000100, $\ldots$. We see that this sequence falls into a cycle of length two. Will this happen for every number?

It is easy to see that a curriculum vitae for a number can be defined iff a life story can be defined. What about sequences? For which numbers can we define an infinite sequence of life stories? By arguments similar to those in lemma \ref{CVSeqdefined} we can see that numbers that allow infinite sequences of life stories are the same as numbers allowing infinite sequences of curricula vitae.

Let us introduce some notation first. Given a number $N$, we build a sequence of life stories denoted as $CLS_n(N)$, where $CLS_0(N) = N$, and the life story of $CLS_n(N)$ is $CLS_{n+1}(N)$. From now on we assume that $N$ is a number that allows an infinite series of life stories.

\begin{lemma}
For $n > 3$, $CLS_n(N)$ has at least 6 zeroes.
\end{lemma}

\begin{proof}
The sum of digits of a life story of a life story, $CLS_2(N)$, is 10. The number of digits is also 10. Hence, such a number cannot have more than 5 different digits. Therefore, its life story, $CLS_2(N)$, has at least 5 zeroes. Hence, the next life story $CLS_3(N)$ has the first digit at least 5. This number too has to have 10 digits that sum up to 10. How many of them are different? We have at least 5 zeroes. We also have that the first digit is not less than 5. Other digits have to sum up to not more than 5. Therefore, we cannot have more than 3 different other digits. Hence, $CLS_n(N)$ has at least 6 zeroes.
\end{proof}

In particular, it means that for $n > 4$, the first digit of $CLS_n(N)$ is at least 6. At this point, it is easy to check all the possibilities proving the following theorem.

\begin{theorem}
All infinite sequences of consecutive life stories end with the cycle 6300000100, 7101001000, 6300000100, $\ldots$.
\end{theorem}

\section{Mutually-praising Pairs}

The numbers 6300000100 and 7101001000 are complete life stories of each other. These numbers are too shy to advertise themselves. But Alice praises Bob, because Bob praises Alice. It's a very advantageous flattery pattern! I will call such a pair a mutually-praising pair. We've already seen a mutually-praising pair of strings: 12 and 001. The new pair of numbers is much better, because they are legitimate numbers.
Two other examples of number pairs thriving on each others' compliments are, first, 130 and 1101, and second, 2210 and 11200.

Obviously, in a legitimate pair of mutually-praising numbers both numbers have zeroes.  Here is an exercise for you. Prove that at least one of them has to end in zero.

Another exercise is to find all mutually-praising number pairs.

\section{What's in the Name?}

Though the sequence A046043 is described as ``autobiographical (or curious)'' numbers, I completely ignored the second name, because ``autobiographical'' is more descriptive than ``curious''. I followed some links in the OEIS, but could not learn where the word ``autobiographical'' came from. 

Martin Gardner gives a puzzle about a ten-digit autobiographical number 6210001000 in his book ``Mathematical Circus'' \cite{Gardner}. Discussing the solution to his puzzle he mentions the name ``tally numbers'' for such numbers.

All autobiographical numbers are calculated by Fred Gavin in \cite{Gavin}. He calls these numbers curious numbers. I would vote against the name curious, and time will show us which name sticks better to this sequence.

\section{Acknowledgements}

I would like to thank OEIS (Online Encyclopedia of Integer Sequences) for being an extremely helpful resource in studying integer sequences. I am thankful to Jane Sherwin, Sergei Bernstein, and Sue Katz for helping me with English for this paper.

\end{document}